\documentclass[a4paper, 12pt]{amsart}
\usepackage{hyperref,upref}

\numberwithin{equation}{section}

\newtheorem{thm}{Theorem}[section]
\newtheorem{lemma}[thm]{Lemma}
\newtheorem{prop}[thm]{Proposition}
\newtheorem{cor}[thm]{Corollary}

\theoremstyle{definition}
\newtheorem{defn}[thm]{Definition}
\newtheorem{ntn}[thm]{Notation}

\theoremstyle{remark}
\newtheorem{eg}[thm]{Example}
\newtheorem{rmk}[thm]{Remark}

\newcommand{\osupp}{\operatorname{osupp}}
\newcommand{\supp}{\operatorname{supp}}
\newcommand{\lsp}{\operatorname{span}}
\newcommand{\clsp}{\overline{\lsp}}

\newcommand{\CC}{\mathbb{C}}

\newcommand{\ZZ}{\mathbb{Z}}

\newcommand{\Bb}{\mathcal{B}}
\newcommand{\Ff}{\mathcal{F}}
\newcommand{\Gg}{\mathcal{G}}
\newcommand{\Kk}{\mathcal{K}}
\newcommand{\Ll}{\mathcal{L}}
\newcommand{\Nn}{\mathcal{N}}
\newcommand{\Oo}{\mathcal{O}}
\newcommand{\Tt}{\mathcal{T}}

\title[Topological graph $C^*$-algebras]{An elementary approach to $C^*$-algebras associated to topological graphs}

\date{15 May 2012}

\author{Hui Li}
\email{hl338@uowmail.edu.au}

\author{David Pask}

\author{Aidan Sims}
\email{dpask, asims@uow.edu.au}

\address{School of Mathematics and Applied Statistics \\
Building 15\\
University of Wollongong\\
Wollongong NSW 2522\\
AUSTRALIA}

\thanks{This research was supported by the Australian Research Council.}

\begin{document}

\begin{abstract}
We develop notions of a representation of a topological graph $E$ and of a covariant representation
of a topological graph $E$ which do not require the machinery of $C^*$-correspondences and
Cuntz-Pimsner algebras. We show that the $C^*$-algebra generated by a universal representation of
$E$ coincides with the Toeplitz algebra of Katsura's topological-graph bimodule, and that the
$C^*$-algebra generated by a universal covariant representation of $E$ coincides with Katsura's
topological graph $C^*$-algebra. We exhibit our results by constructing the isomorphism between the
$C^*$-algebra of a row-finite directed graph $E$ with no sources and the $C^*$-algebra of the
topological graph arising from the shift map acting on infinite path space $E^\infty$.
\end{abstract}

\maketitle

\section{Introduction}\label{sec:intro}

Let $E$ be a countable directed graph with vertex set $E^0$, edge set $E^1$ and range and source
maps $r, s : E^1 \to E^0$. The Toeplitz-Cuntz-Krieger algebra $\Tt C^*(E)$ is the universal
$C^*$-algebra generated by a family of mutually orthogonal projections $\{p_v : v \in E^0\}$ and a
family of partial isometries $\{s_e : e \in E^1\}$ such that
\begin{enumerate}
\item $s^*_e s_e = p_{s(e)}$ for every $e \in E^1$, and
\item $p_v \ge \sum_{e \in F} s_e s^*_e$ for every $v \in E^0$ and finite $F \subset
    r^{-1}(v)$.
\end{enumerate}
The Cuntz-Krieger algebra $C^*(E)$ is universal for families as above satisfying the additional
relation that $p_v = \sum_{r(e) = v} s_e s^*_e$ whenever $r^{-1}(v)$ is nonempty and finite.

These $C^*$-algebras have been studied very extensively over the last fifteen years, part of their
appeal being precisely that they can be defined, as above, in the space of a paragraph. Combined
with key structure theorems like the gauge-invariant uniqueness theorem and the Cuntz-Krieger
uniqueness theorem \cite{BatesHongEtAl:IJM02}, or Fowler, Laca and Raeburn's analogue of Coburn's
Theorem for the Toeplitz algebra of a directed graph \cite{FowlerRaeburn:IUMJ99}, the above
elementary presentation often makes it easy to verify that a given $C^*$-algebra is isomorphic to
$C^*(E)$ or $\Tt C^*(E)$ as appropriate.

Some years after the introduction of graph $C^*$-algebras, Katsura introduced topological graphs
and their $C^*$-algebras \cite{Katsura:TAMS04}. His construction is based on that of
\cite{Katsura:JFA04}, which in turn was a modification of Pimsner's construction of $C^*$-algebras
associated to $C^*$-correspondences in \cite{Pimsner:FIC97}. Roughly speaking, a topological graph
is like a directed graph except that $E^0$ and $E^1$ are locally compact Hausdorff spaces rather
than countable discrete sets, and $r$ and $s$ are required to be topologically well-behaved.
Building on Fowler and Raeburn's construction of a $C^*$-correspondence from a directed graph given
in \cite{FowlerRaeburn:IUMJ99}, Katsura associated to each topological graph $E$ a
$C^*$-correspondence $X(E)$, and defined the $C^*$-algebra of the topological graph $E$ to be the
$C^*$-algebra $\Oo_{X(E)}$ he had associated to the module $X(E)$ in \cite{Katsura:JFA04}.

The drawback of this approach is that the relations defining the $C^*$-algebra are quite
complicated and cannot be stated without first introducing at least the rudiments of
$C^*$-correspondences, which is quite a bit of overhead --- see Section~\ref{sec:Katsura}. Clearly
it would be handy to be able to circumvent all this technical overhead when defining and dealing
with the $C^*$-algebras of topological graphs. This paper takes the first step in this direction.
Our main theorems show that if $E$ is a topological graph, then the Toeplitz algebra $\Tt(E)$ and
Katsura's topological graph algebra $\Oo(E)$ of \cite{Katsura:TAMS04} can be described as
$C^*$-algebras universal for relatively elementary relations involving elements of $C_0(E^0)$ and
$C_c(E^1)$. In particular, our presentations can be stated without any reference to
$C^*$-correspondences. We state our results so that they apply to arbitrary topological graphs, but
we also show how our crucial covariance condition simplifies when the range map $r : E^1 \to E^0$
is a local homeomorphism and $r(E^1)$ is closed. Even this special situation is interesting: it
includes, for example, Cantor minimal systems, and all crossed products of abelian $C^*$-algebras
by $\ZZ$.

To emphasise how little background is needed to present our definition of a representation of a
topological graph and our covariance condition, we begin by stating all our key definitions and
main theorems in Section~\ref{sec:main results}. In Section~\ref{sec:Katsura} we recall Katsura's
construction of a $C^*$-correspondence from a topological graph and his definitions of the
associated $C^*$-algebras. In Section~\ref{sec:proofs}, we prove our main results. We finish in
Section~\ref{sec:example} by considering the topological graphs $\widehat{E}$ obtained from
row-finite directed graphs $E$ with no sources by taking $\widehat{E}^0 = \widehat{E}^1 = E^\infty$
and defining the range map to be the identity map, and the source map to be the left-shift map
$\sigma$. We apply our results to provide a relatively elementary proof that $\Oo(\widehat{E})$ is
canonically isomorphic to $C^*(E)$. This result could be recovered from \cite{Brownlowe, BRV}, but
working out the details provides a good example of the efficacy of our results.

\section{Main results}\label{sec:main results}

\begin{defn}[{\cite[Definition~2.1]{Katsura:TAMS04}}]
A quadruple $E=(E^0,E^1,r,s)$ is called a \emph{topological graph} if $E^0$, $E^1$ are locally
compact Hausdorff spaces, $r:E^1 \to E^0$ is a continuous map, and $s:E^1\to E^0$ is a local
homeomorphism.
\end{defn}

We think of $E^0$ as a space of vertices, and we think of each $e \in E^1$ as an arrow pointing
from $s(e)$ to $r(e)$. If $E^0$, $E^1$ are both countable and discrete, then $E$ is a directed
graph in the sense of \cite{KumjianPaskEtAl:JFA97, Raeburn:GraphAlg}.

Given $x \in C_c(E^1)$ and $f \in C_0(E^0)$, we define $x \cdot f$ and $f \cdot x$ in $C_c(E^1)$ by
\begin{equation}\label{eq:actions}
    (x\cdot f)(e)=x(e)f(s(e))\quad\text{ and } (f \cdot x)(e)=f(r(e))x(e).
\end{equation}

To define our notion of a representation of a topological graph, we need a couple of preliminary
ideas. An \emph{$s$-section} in a topological graph $E$ is a subset $U \subset E^1$ such that $s
\vert_U$ is a homeomorphism. An \emph{$r$-section} is defined similarly, and a \emph{bisection} is a set which is both an $s$-section and an
$r$-section.

If $x \in C_c(E^1)$ then $\osupp(x)$ denotes the precompact open set $\{e \in E^1 : x(e) \not=
0\}$, and $\supp(x)$ is the closure of $\osupp(x)$. If $x \in C_c(E^1)$ and $\osupp(x)$ is an
$s$-section, we define $\widehat{x} : E^0 \to \mathbb{C}$ by
\begin{equation}\label{eq:xhat}
\widehat{x}(s(e)) := \vert x(e) \vert^2 \text{ for $e \in \osupp(x)$, and }
    \widehat{x}(v) = 0 \text{ for $v \notin s(\osupp(x))$}.
\end{equation}

\begin{defn}\label{reduced Toep rep}
Let $E$ be a topological graph. A \emph{Toeplitz representation of $E$} in a $C^*$-algebra $B$ is a
pair $(\psi, \pi)$ where $\psi: C_c(E^1) \to B$ is a linear map, $\pi:C_0(E^0) \to B$ is a
homomorphism, and
\begin{enumerate}
\item\label{it:actions} $\psi(f \cdot x)=\pi(f)\psi(x)$, for all $x \in C_c(E^1)$, $f \in
    C_0(E^0)$;
\item\label{it:s*s} for $x \in C_c(E^1)$ such that $\supp(x)$ is contained in an open
    $s$-section, $\pi(\widehat{x})=\psi(x)^* \psi(x)$; and
\item\label{it:s*t} for $x$, $y \in C_c(E^1)$ such that $\supp(x)$ and $\supp(y)$ are contained
    in disjoint open $s$-sections, $\psi(x)^* \psi(y)=0$.
\end{enumerate}
We say that a Toeplitz representation $(\psi, \pi)$ of $E$ in $B$ is \emph{universal} if for any
Toeplitz representation $(\psi',\pi')$ of $E$ in $C$, there exists a homomorphism $h : B \to C$,
such that $h \circ \psi = \psi'$ and $h \circ \pi = \pi'$.
\end{defn}

\begin{rmk}\label{alternative def of reduced toep rep}
Suppose that $x \in C_c(E^1)$ and $\supp(x)$ is contained in an open $s$-section $U$. Since $E^1$
is locally compact, we may cover $\supp(x)$ with precompact open sets $\{V_i : i \in I\}$, and
since $\supp(x)$ is compact, we may pass to a finite subcover $\{V_i : i \in F\}$. Then each $V_i
\cap U$ is precompact and open, so $\bigcup_{i \in F} (V_i \cap U)$ is a precompact open
$s$-section containing $\supp(x)$. Thus we may assume in Conditions
\eqref{it:s*s}~and~\eqref{it:s*t} that the open $s$-sections containing $\supp(x)$ and $\supp(y)$
are precompact.
\end{rmk}

To state our first main theorem, recall that if $E$ is a topological graph, then $\mathcal{T}(E)$
denotes the Toeplitz algebra of Katsura's topological-graph bimodule (see Notation~\ref{def of T(E)
and O(E)}).

\begin{thm}\label{thm:mainToeplitz}
Let $E$ be a topological graph. Then there is a universal Toeplitz representation $(i_1, i_0)$ of
$E$ which generates $\mathcal{T}(E)$. Moreover the $C^*$-algebra generated by the image of any universal Toeplitz representation of $E$ is isomorphic to $\mathcal{T}(E)$.
\end{thm}

\begin{rmk}
Since the map $\psi$ occurring in a Toeplitz representation of $E$ is not a homomorphism, it will
not usually be norm-decreasing with respect to the supremum norm. So it is not clear that one can
just extend by continuity a linear map $\psi_0$ defined on a dense subspace of $C_c(E^1)$. We show
in Proposition~\ref{cor:checkable toeplitz} how to get around this difficulty: the map $\psi$
\emph{is} norm-decreasing with respect to the supremum norm when applied to functions supported on
$s$-sections.
\end{rmk}

We now describe the covariance condition for a Toeplitz representation of a topological graph. The
condition is somewhat technical, but we will indicate how it simplifies under additional hypotheses
in Corollary~\ref{covariant condition of top graph with local homeomorphic range map}.

We first need a little notation from \cite{Katsura:TAMS04}.

\begin{defn}[{\cite[Definition~2.6]{Katsura:TAMS04}}]
Let $E$ be a topological graph. We define
\begin{enumerate}
\item $E_{\mathrm{sce}}^0=E^0 \setminus \overline{r(E^1)}$.
\item $E_{\mathrm{fin}}^0 = \{v \in E^0 : {}$ there exists a neighbourhood $N$  of $v$ such
    that $r^{-1}(N)$ is compact$\}$.
\item $E_{\mathrm{rg}}^0=E_{\mathrm{fin}}^0 \setminus \overline{E_{\mathrm{sce}}^0}$.
\end{enumerate}
\end{defn}

\begin{rmk}\label{properties of special subsets of E^0}
The set $E_{\mathrm{fin}}^0$ is open in $E^0$. To see this, fix $v \in E_{\mathrm{fin}}^0$. Then
there exists a neighbourhood $U$ of $v$ such that $r^{-1}(U)$ is compact and there exists an open
neighbourhood $V$ of $v$ contained in $U$. So $V \subset E_{\mathrm{fin}}^0$, whence
$E_{\mathrm{fin}}^0$ is open. It follows that $E_{\mathrm{sce}}^0$, $E_{\mathrm{fin}}^0$,
$E_{\mathrm{rg}}^0$ are all open in $E^0$. Moreover $E_{\mathrm{rg}}^0$ is the intersection of
$E_{\mathrm{fin}}^0$ with the interior of $\overline{r(E^1)}$. Finally, as proved by Katsura (see
\cite[Lemma~1.23]{Katsura:TAMS04}), for any compact subset $K \subset E_{\mathrm{fin}}^0$, the set
$r^{-1}(K)$ is compact in $E^1$.
\end{rmk}

\begin{ntn}
Let $X$ be a locally compact Hausdorff space and $U$ be an open subset of $X$. Then the standard
embedding $\iota_U$ of $C_0(U)$ as an ideal of $C_0(X)$ is given by
\[
\iota_U(f)(x) := \begin{cases} f(x) &\text{ if $x \in U$} \\ 0 &\text{ otherwise.}\end{cases}
\]
We will usually suppress the map $\iota_U$ and just identify each of $C_0(U)$ and $C_c(U)$ with
their images in $C_0(X)$ under $\iota_U$. That is, we think of $C_0(U)$ as an ideal of $C_0(X)$,
and we regard $C_c(U)$ as an (algebraic) ideal of $C_0(X)$.
\end{ntn}

\begin{rmk}\label{POU on subspace}
Let $X$ be a Hausdorff space, and let $K \subset X$ be compact. Fix a finite cover $\Nn$ of $K$ by
open subsets of $X$. By \cite[Chapter~5.W]{Kelley:GenTop}, there exists a partition of unity $\{h_N
: N \in \Nn\}$ on $K$ subordinate to $\{N \cap K: N \in \mathcal{N} \}$, where $h_N(x) \in [0,1]$,
for all $N \in \mathcal{N}$, $x \in K$. Since $X$ may not be normal, we cannot necessarily extend
this to a partition of unity on $X$. Nevertheless, if $f \in C_c(X)$ with $\supp(f) \subset K$,
then the functions $f_N : X \to \CC$ given by
\begin{equation}\label{eq:f_N}
f_N(x) := \begin{cases}
    f(x) h_N(x) &\text{ if $x \in K$} \\
    0 &\text{ otherwise}
\end{cases}
\end{equation}
belong to $C_c(X)$. Hence each $\osupp(f_N) \subset N \cap K$, and $\sum_{N \in \mathcal{N}}f_N=
f$.
\end{rmk}

\begin{defn}\label{dfn:new covariance}
Let $E$ be a topological graph, and let $(\psi,\pi)$ be a Toeplitz representation of $E$ in a
$C^*$-algebra $B$. We call $(\psi,\pi)$ \emph{covariant} if there exists a collection $\mathcal{G}
\subset C_c(E^0_{\textrm{rg}})$ of nonnegative functions which generates $C_0(E^0_{\textrm{rg}})$
as a $C^*$-algebra, and for each $f \in \mathcal{G}$ there exist a finite cover $\Nn_f$ of
$r^{-1}(\supp(f))$ by open $s$-sections, and a collection of nonnegative continuous functions
$\{f_N: N \in \mathcal{N}_f\}$ such that
\begin{enumerate}
\item\label{osupp(f_N) is contained in N intersects r^{-1}(supp(f))} $\osupp(f_N) \subset N \cap r^{-1}(\supp(f))$ for all $N \in \mathcal{N}_f$;
\item\label{sum of f_N=f circ r} $\sum_{N \in \mathcal{N}_f}f_N=f \circ r$; and
\item\label{CK-relation} $\pi(f) = \sum_{N \in \mathcal{N}_f} \psi\big(\sqrt{f_N}\big) \psi\big(\sqrt{f_N}\big)^*$.
\end{enumerate}
We say that a covariant Toeplitz representation $(\psi, \pi)$ of $E$ in $B$ is \emph{universal} if
for any covariant Toeplitz representation $(\psi',\pi')$ of $E$ in $C$, there exists a homomorphism
$h : B \to C$, such that $h \circ \psi = \psi'$ and $h \circ \pi = \pi'$.
\end{defn}

\begin{rmk}
Definition~\ref{dfn:new covariance} is formulated so as to make it easy to check that a given pair
$(\psi,\pi)$ is covariant. However, when \emph{using} covariance of a pair $(\psi,\pi)$ it is
helpful to know that \eqref{CK-relation} of Definition~\ref{dfn:new covariance} holds for every $f
\in C_c(E^0_{\textrm{rg}})$, every $\mathcal{N}_f$, and every $\{f_N:\ N \in \mathcal{N}_f \}$
satisfying \eqref{osupp(f_N) is contained in N intersects r^{-1}(supp(f))} and \eqref{sum of f_N=f
circ r} of Definition~\ref{dfn:new covariance}. We prove this in Proposition~\ref{reduced covariant
Toep rep is covariant Toep rep}.
\end{rmk}

\begin{rmk}
In Definition~\ref{dfn:new covariance}, we implicitly used Remark~\ref{POU on subspace} because
$r^{-1}(\supp(f))$ is compact by Remark~\ref{properties of special subsets of E^0}. Observe that
the covariance condition for a Toeplitz representation of $E$ only involves functions in
$C_0(E^0_{\textrm{rg}}) \subset C_0(E^0)$.
\end{rmk}

To state our second main theorem, recall that if $E$ is a topological graph, then $\mathcal{O}(E)$
denotes Katsura's topological graph $C^*$-algebra (see Notation~\ref{def of T(E) and O(E)}).

\begin{thm}\label{thm:mainCK}
Let $E$ be a topological graph. Then there is a universal covariant Toeplitz representation
$(j_1,j_0)$ of $E$ which generates $\mathcal{O}(E)$. Moreover the $C^*$-algebra generated by the
image of any universal covariant Toeplitz representation of $E$ is isomorphic $\mathcal{O}(E)$.
\end{thm}

Although Definition~\ref{dfn:new covariance} looks complicated, the hypotheses are relatively easy
to check in specific instances. To give some intuition, we indicate how the definition simplifies
if the range map $r : E^1 \to E^0$ is a local homeomorphism and $r(E^1)$ is closed. This situation
still includes many interesting examples.

\begin{defn}\label{def:r-fibration}
Let $E$ be a topological graph. A pair $(\mathcal{U}, V)$ is called a \emph{local $r$-fibration} if
$V$ is a subset of $E^0$, and $\mathcal{U}$ is a finite collection of mutually disjoint bisections
such that $r^{-1}(V)=\bigcup_{U \in \mathcal{U}}U$, and $r(U)=V$ for each $U \in \mathcal U$. A
local $r$-fibration is \emph{precompact} if each $U \in \mathcal{U}$ is precompact and $V$ is
precompact. A local $r$-fibration is \emph{open} if each $U \in \mathcal{U}$ is open and $V$ is
open.
\end{defn}

Suppose that $(\mathcal{U}, V)$ is an open local $r$-fibration. Suppose that $U \in \mathcal{U}$ and that
$f \in C_c(V)$. We write $r^*_U f \in C_c(U)$ for the function
\[
    r^*_Uf : e \mapsto \begin{cases}
        f(r(e)) &\text{ if $e \in U$}\\
        0 &\text{ otherwise}.
    \end{cases}
\]

\begin{cor}\label{covariant condition of top graph with local homeomorphic range map}
Let $E$ be a topological graph. Suppose that $r$ is a local homeomorphism and $r(E^1)$ is closed. Let $(\psi,\pi)$ be a Toeplitz representation of $E$. Then $(\psi,\pi)$ is covariant if and only if there exists a collection $\mathcal{G} \subset C_c(E^0_{\textrm{rg}})$ of nonnegative functions  which generates $C_0(E^0_{\textrm{rg}})$ as a $C^*$-algebra, and for each $f \in \mathcal{G}$ there exists an open local $r$-fibration $(\mathcal{U},V)$ such that $\supp(f) \subset V$, and
\begin{equation}\label{CK-relation when r is loc homeo and r(E^1) is closed}
\pi(f) = \sum_{U \in \mathcal{U}} \psi\big(\sqrt{r^*_U f}\big) \psi\big(\sqrt{r^*_U f}\big)^*.
\end{equation}
If $(\psi,\pi)$ is covariant, then \emph{Equation~\ref{CK-relation when r is loc homeo and r(E^1) is closed}} holds for every $f \in C_c(E^0_{\textrm{rg}})$ and open local $r$-fibration $(\mathcal{U},V)$ with $\supp(f) \subset V$.
\end{cor}

\begin{rmk}\label{converse statement proof of covariant condition}
Let $E$ be a topological graph. Fix $f \in C_c(E^0_{\textrm{rg}})$ and suppose that
$(\mathcal{U},V)$ is an open local $r$-fibration such that $\supp(f) \subset V$. By
Remark~\ref{properties of special subsets of E^0}, $r^{-1}(\supp(f))$ is compact. Since
$\mathcal{U}$ is an open cover of $r^{-1}(\supp(f))$ by disjoint open bisections, we apply Remark~\ref{POU on subspace} to obtain functions $\{f_U: U \in \mathcal{U}\}$ such that $\osupp(f_U) \subset U  \cap r^{-1}(\supp(f))$, and $\sum_{U \in \mathcal{U}}f_U=f \circ r$, for all $U \in \mathcal{U}$. We then have $r^*_U f = f_U$, for all $U \in \mathcal{U}$.
\end{rmk}

\section{\texorpdfstring{$C^*$}{C*}-correspondences, \texorpdfstring{$C^*$}{C*}-algebras and Katsura's construction}\label{sec:Katsura}

We recall some background on Hilbert $C^*$-modules. For more detail see \cite{RaeburnWilliams:Morita}.

\begin{defn}
Let $A$ be a $C^*$-algebra. A \emph{right Hilbert $A$-module} is a right $A$-module $X$ equipped
with a map $\langle \cdot, \cdot \rangle_A : X \times X \to A$ such that for $x,y \in X$, $a \in A$,
\begin{enumerate}
    \item $\langle x,x \rangle_A \geq 0$ with equality only if $x = 0$;
    \item \label{A-linear in the second variable} $\langle x,y \cdot a\rangle_A=\langle x,y\rangle_Aa$;
    \item\label{it:adjoint of ip}$\langle x,y\rangle_A = \langle y,x\rangle_A^*$; and
    \item $X$ is complete in the norm\footnote{It is not supposed to be immediately obvious
        that this defines a norm, but it is true (see \cite{RaeburnWilliams:Morita}).} defined
        by $\|x\|_A ^2 = \|\langle x, x\rangle_A\|$.
\end{enumerate}

Recall from, for example, \cite[Page~72]{Raeburn:GraphAlg}, that a right Hilbert $A$-module $X$ is
a \emph{$C^*$-correspondence} over $A$ if there is a left action of $A$ on $X$ such that
\begin{equation}\label{eq:adjointable left action}
    \langle a^* \cdot y,x \rangle_A=\langle y,a\cdot x\rangle_A \quad\text{ for all $a \in A$, $x$, $y \in X$.}
\end{equation}
\end{defn}

An operator $T:X \to X$ is \emph{adjointable} if there exists $T^*:X \to X$ such that $\langle
T^*y,x\rangle_A=\langle y,Tx\rangle_A$, for all $x$, $y \in X$. The adjoint $T^*$ is unique, and
$T$ is automatically bounded and linear. The set $\mathcal{L}(X)$ of adjointable operators on $X$
is a $C^*$-algebra. Equation~\eqref{eq:adjointable left action} implies that there is a
homomorphism $\phi : A \to \Ll(X)$ such that $\phi(a)x = a\cdot x$ for all $a \in A$ and $x \in X$.

\begin{defn}
Let $A$ be a $C^*$-algebra and let $X$ be a right Hilbert $A$-module. Fix $x$, $y \in X$, we define
$\Theta_{x,y}:X \to X$ by $\Theta_{x,y}(z)=x\cdot \langle y,z \rangle_A$, for all $z \in X$. We
define $\Kk(X) := \clsp\{\Theta_{x,y} : x,y \in X\}$.
\end{defn}
A calculation shows that $T\Theta_{x,y}=\Theta_{Tx,y}$, and $\Theta_{x,y}^*=\Theta_{y,x}$, for all $x$, $y
\in X$, $T \in \mathcal{L}(X)$. Hence $\mathcal{K}(X)$ is a closed two-sided ideal of $\mathcal{L}(X)$.

Toeplitz representations and Toeplitz algebras of $C^*$-correspondences were introduced and studied
in \cite{Pimsner:FIC97}. Cuntz-Pimsner algebras were also introduced in \cite{Pimsner:FIC97}, but
the definition was later modified by Katsura in \cite[Definition~3.5]{Katsura:JFA04} so as to include
graph algebras as a special case \cite[Example~8.13]{Raeburn:GraphAlg}. We use Katsura's definition in this paper.

\begin{defn}[{\cite[Definition~2.1]{Katsura:JFA04}}]\label{def of Toep rep of correspondence}
Let $A$ be a $C^*$-algebra and let $X$ be a $C^*$-cor\-res\-pond\-ence over $A$. A \emph{Toeplitz
representation} of $X$ in a $C^*$-algebra $B$ is a pair $(\psi,\pi)$, where $\psi:X \to B$ is a
linear map, $\pi:A \to B$ is a homomorphism, and for any $x$, $y \in X$ and $a \in A$,
\begin{enumerate}
    \item $\psi(a \cdot x)=\pi(a)\psi(x)$; and
    \item $\psi(x)^*\psi(y)=\pi(\langle x,y \rangle_A)$.
\end{enumerate}
As in \cite[Proposition~1.3]{FowlerRaeburn:IUMJ99}, we say that Toeplitz representation
$(\psi,\pi)$ of $X$ in $B$ is \emph{universal} if for any Toeplitz representation $(\psi',\pi')$ of
$X$ in $C$, there exists a homomorphism $h : B \to C$, such that $h \circ \psi=\psi'$, and $h \circ
\pi = \pi'$.
\end{defn}

\begin{rmk}[{\cite[Page~370]{Katsura:JFA04}}]\label{property of Toep rep}
Let $X$ be a $C^*$-correspondence over a $C^*$-algebra $A$ and let $(\psi,\pi)$ be a Toeplitz
representation of $X$. Condition~$(2)$ of Definition~\ref{def of Toep rep of correspondence} and
that $\pi$ is a homomorphism imply that $\psi(x \cdot a)=\psi(x)\pi(a)$, for all $x \in X$, $a \in
A$, and also that $\psi$ is bounded with $\Vert \psi \Vert \leq 1$.
\end{rmk}

Proposition~1.3 of \cite{FowlerRaeburn:IUMJ99} implies that there exists a $C^*$-algebra
$\mathcal{T}_X$ generated by the image of a universal Toeplitz representation $(i_X,i_A)$ of $X$.
This $C^*$-algebra is unique up to canonical isomorphism and we call it the \emph{Toeplitz algebra}
of $X$. Given another Toeplitz representation $(\psi,\pi)$ of $X$ in a $C^*$-algebra $B$, we write
$h_{\psi,\pi}$ for the induced homomorphism from $\Tt_X$ to $B$.

Recall from \cite[Page~202]{Pimsner:FIC97} (see also \cite[Proposition~1.6]{FowlerRaeburn:IUMJ99})
that if $(\psi,\pi)$ is a Toeplitz representation of a $C^*$-correspondence $X$ in a $C^*$-algebra
$B$, then there is a unique homomorphism $\psi^{(1)} : \Kk(X) \to B$ such that
$\psi^{(1)}(\Theta_{x,y}) =\psi(x)\psi(y)^*$ for all $x$, $y \in X$.

Recall also that given a $C^*$-algebra $A$ and a closed two-sided ideal $J$ of $A$, we define
$J^{\perp} := \{a \in A:ab=0\text{ for all }b \in J\}$. Then $J^{\perp}$ is also a closed two-sided
ideal of $A$.

\begin{defn}[{\cite[Definitions~3.4 and 3.5]{Katsura:JFA04}}]\label{dfn:O_X}
Let $A$ be a $C^*$-algebra, let $X$ be a $C^*$-correspondence over $A$, and write $\phi : A \to
\Ll(X)$ for the homomorphism implementing the left action. A Toeplitz representation $(\psi,\pi)$
of $X$ is \emph{covariant} if $\psi^{(1)}(\phi(a))=\pi(a)$ for all $a \in \phi^{-1}(\Kk(X)) \cap
(\ker\phi)^{\perp}$.

A covariant Toeplitz representation $(\psi,\pi)$ of $X$ in $B$ is \emph{universal} if for any
covariant Toeplitz representation $(\psi',\pi')$ of $X$ in $C$, there exists a homomorphism $h$
from $B$ into $C$, such that $h \circ \psi=\psi'$, and $h \circ \pi = \pi'$.
\end{defn}

Recall from \cite[Page~75]{Raeburn:GraphAlg} that there exists a $C^*$-algebra $\mathcal{O}_X$
generated by the image of a universal covariant Toeplitz representation $(j_X,j_A)$ of $X$. This
$C^*$-algebra is unique up to canonical isomorphism. Given another covariant Toeplitz
representation $(\psi,\pi)$ of $X$ in a $C^*$-algebra $B$, we write $\psi\times \pi$ for the
induced homomorphism from $\mathcal{O}_X$ to $B$.

Let $E$ be a topological graph. As in \cite{Katsura:TAMS04}, define right and left actions of
$C_0(E^0)$ on $C_c(E^1)$ by Equation~\eqref{eq:actions}. For $x_1$, $x_2 \in C_c(E^1)$, define $\langle
x_1,x_2 \rangle_{C_0(E^0)}: E^0 \to \mathbb{C}$ by
\[
\langle x_1,x_2 \rangle_{C_0(E^0)}(v)=\sum_{s(e)=v}\overline{x_1(e)}x_2(e).
\]
If $s^{-1}(v)=\emptyset$, then our convention is that the sum is equal to $0$. If $\osupp(x)$ is an
$s$-section, then the function $\widehat{x}$ of Equation~\eqref{eq:xhat} is equal to $\langle x,x
\rangle_{C_0(E^0)}$. As in \cite{Katsura:TAMS04} (see also \cite[Page~79]{Raeburn:GraphAlg}),
$\langle\cdot,\cdot \rangle_{C_0(E^0)}$ defines a $C_0(E^0)$-valued inner product on $C_c(E^1)$,
and the completion $X(E)$ of $C_c(E^1)$ in the norm $\|x\|_{C_0(E^0)}^2 = \|\langle x, x
\rangle_{C_0(E^0)}\|$ is a $C^*$-correspondence over $C_0(E^0)$, which we call the \emph{graph
correspondence} associated to $E$.

\begin{ntn}\label{def of T(E) and O(E)}
We denote by $\mathcal{T}(E)$ \cite[Definition~2.2]{Katsura:TAMS04} the Toeplitz algebra
$\mathcal{T}_{X(E)}$, and we denote by $\mathcal{O}(E)$ \cite[Definition~2.10]{Katsura:TAMS04} the
$C^*$-algebra $\mathcal{O}_{X(E)}$.
\end{ntn}

\section{Proofs of the main results}\label{sec:proofs}
To prove Theorem~\ref{thm:mainToeplitz}, we must show that $\psi(x)^* \psi(y) = \pi(\langle x,
y\rangle_{C_0(E^0)})$ for all $x,y \in C_c(E^1)$. We establish this formula for $x,y$ supported on
$s$-sections in Lemma~\ref{polar identity}, and then extend it to arbitrary $x, y \in C_c(E^1)$ in
Proposition~\ref{property of reduced Toep rep}.

Let $U$, $V$ be complex vector spaces. Then any sesquilinear form $\varphi:V \times V \to U$ which
is conjugate linear in the first variable satisfies the polarisation identity
\[
\varphi(v_1,v_2)=\frac{1}{4}\sum_{n=0}^{3}(-i)^n\varphi(v_1+i^n v_2,v_1+i^nv_2).
\]
To prove this, just expand the sum.

\begin{lemma}\label{polar identity}
Let $E$ be a topological graph and let $(\psi,\pi)$ be a Toeplitz representation of $E$. Fix $x_1$,
$x_2 \in C_c(E^1)$. Suppose that $\mathrm{supp}(x_1) \cup \mathrm{supp}(x_2)$ is contained in an
open $s$-section. Then $\pi(\langle x_1,x_2 \rangle_{C_0(E^0)})=\psi(x_1)^* \psi(x_2)$.
\end{lemma}
\begin{proof}
The polarisation identity for $\langle \cdot ,\cdot \rangle_{C_0(E^0)}$, together with
Definition~\ref{reduced Toep rep}\eqref{it:s*s}, gives
\begin{flalign*}
&&\pi(\langle x_1,x_2 \rangle_{C_0(E^0)})
    &=\frac{1}{4}\sum_{n=0}^{3}(-i)^n\pi(\langle x_1+i^nx_2,x_1+i^nx_2\rangle_{C_0(E^0)})& \\
    &&&=\frac{1}{4}\sum_{n=0}^{3}(-i)^n\psi(x_1+i^nx_2)^*\psi(x_1+i^nx_2) =\psi(x_1)^* \psi(x_2). &\qedhere
\end{flalign*}
\end{proof}

In the following and throughout the rest of the paper we write $U^c$ for the complement $X
\setminus U$ of a subset $U$ of a set $X$.

\begin{rmk}\label{property of regular space}
Let $X$ be locally compact Hausdorff space. Fix a compact subset $K \subset X$ and an open
neighbourhood $U$ of $K$. Since $U^c$ is closed and disjoint from $K$, there is a function $f \in
C(X,[0,1])$ which is identically 1 on $K$ and identically 0 on $U^c$ (see for example
\cite[Theorem~37.A]{Simmons:TopAnal}). So $V := f^{-1}\big((1/2, 3/2)\big)$ is open and satisfies
$K \subset V \subset \overline{V} \subset U$.
\end{rmk}

\begin{prop}\label{property of reduced Toep rep}
Let $E$ be a topological graph and let $(\psi,\pi)$ be a Toeplitz representation of $E$. Fix open
$s$-sections $U_1$, $U_2 \subset E^1$, and $x_1$, $x_2 \in C_c(E^1)$ with $\mathrm{supp}(x_i)
\subset U_i$, for $i=1$, $2$. Then $\pi(\langle x_1,x_2 \rangle_{C_0(E^0)})=\psi(x_1)^* \psi(x_2)$.
\end{prop}
\begin{proof}
Let $K_i=\supp(x_i)$, for $i=1$, $2$, and $K=K_1 \cup K_2$. Since $E^1$ is locally compact and
Hausdorff, Remark~\ref{property of regular space} implies that there exist open sets $U_i '\subset
E^1$ such that $K_i \subset U_i' \subset \overline{U_i'} \subset U_i$, for $i=1$, $2$.

For each $e \in K$ and $W \in \{U_i, U_i', \overline{U_i'}^c : i=1, 2\}$ such that $e \in W$, fix
an open neighbourhood $V(e,W)$ of $e$ such that $\overline{V(e,W)} \subset W$. Define
\begin{align*}
N_e=\Big( \bigcap_{e \in U_i}V(e,U_i) \Big)
        \bigcap \Big( \bigcap_{e \in K_i}V(e,U_i') \Big)
        \bigcap \Big( \bigcap_{e \in U_i^c} V(e, \overline{U_i'}^{c}) \Big).
\end{align*}
Then for $i=1$, $2$,
\begin{enumerate}
\item if $e \in U_i$, then $\overline{N_e} \subset U_i$;
\item if $e \in K_i$, then $\overline{N_e} \subset U_i'$; and
\item if $e \notin U_i$, then $\overline{N_e} \cap \overline{U_i'}=\emptyset$.
\end{enumerate}
Since $K$ is compact, there is a finite subset $F \subset K$, such that $\{N_e : e \in F\}$ covers
$K$.

Fix $e$, $f \in F$. Suppose that $\overline{N_e} \cap \overline{N_f} \neq \emptyset$. We claim that
\begin{equation}\label{union of supports contained in the same s-section}
\text{either } \overline{N_e} \cup \overline{N_f} \subset U_1,\text{ or }
    \overline{N_e} \cup \overline{N_f} \subset U_2.
\end{equation}
When $e=f$, this is trivial. Suppose $e \neq f$. Assume without loss of generality that $e \in
K_1$. Then $(2)$ forces $\overline{N_e} \subset U_1'$, so $\overline{N_e} \cap \overline{N_f} \neq
\emptyset$ forces $\overline{N_f} \cap \overline{U_1'} \neq \emptyset$. $(3)$ then forces $f \in
U_1$, so $(1)$ forces $\overline{N_f} \subset U_1$ and hence $\overline{N_e} \cup \overline{N_f}
\subset U_1$ as required.

Since $\{N_e : e \in F\}$ is a finite open cover of $K$, Remark~\ref{POU on subspace} implies that
there are finite collections of functions $\{x_{1,e} : e \in F \}$, $\{x_{2,e} : e \in F \} \subset
C_c(E^1)$ such that $\osupp(x_{i,e}) \subset N_e \cap K$ for all $e \in F$, $i=1$, $2$, and
$\sum_{e \in F}x_{1,e}=x_1$, $\sum_{e \in F}x_{2,e}=x_2$. Linearity of $\psi$ gives
$\psi(x_1)^*\psi(x_2)=\sum_{e,f \in F}\psi(x_{1,e})^*\psi(x_{2,f})$. Fix $e$, $f \in F$ such that
$\overline{N_e} \cap \overline{N_f} = \emptyset$. By Remark~\ref{property of regular space} there
are disjoint open $s$-sections $O_e$, $O_f \subset E^1$, such that $\overline{N_e} \subset O_e$,
and $\overline{N_f} \subset O_f$. Thus condition~\eqref{it:s*t} of Definition~\ref{reduced Toep
rep} gives $\psi(x_{1,e})^*\psi(x_{2,f})=0$ since $\supp(x_{i,e}) \subset \overline{N_e}$ for all
$e \in F$, $i=1$, $2$. It follows that $\psi(x_1)^*\psi(x_2)=\sum_{\overline{N_e} \cap
\overline{N_f} \neq \emptyset}\psi(x_{1,e})^*\psi(x_{2,f})$. Whenever $\overline{N_e} \cap
\overline{N_f} \neq \emptyset$, Equation~\eqref{union of supports contained in the same s-section}
implies that $\overline{N_e} \cup \overline{N_f}$ is contained in an open $s$-section, so
Lemma~\ref{polar identity} gives $\psi(x_1)^*\psi(x_2)=\sum_{\overline{N_e} \cap \overline{N_f}
\neq \emptyset}\pi(\langle x_{1,e},x_{2,f} \rangle_{C_0(E^0)})=\pi(\langle
x_1,x_2\rangle_{C_0(E^0)})$.
\end{proof}

\begin{prop}\label{reduced Toep rep is Toep rep}
Let $E$ be a topological graph and let $(\psi,\pi)$ be a Toeplitz representation of $E$. Then
$\psi$ is a bounded linear operator on $(C_c(E^1),\Vert\cdot\Vert_{C_0(E^0)})$. Let
$\widetilde{\psi}$ be the unique extension of $\psi$ to $X(E)$. Then the pair
$(\widetilde{\psi},\pi)$ is a Toeplitz representation of $X(E)$.
\end{prop}
\begin{proof}
Fix $x_1$, $x_2 \in C_c(E^1)$. Let $K=\supp(x_1) \cup \supp(x_2)$. For each $e \in K$, there exists
an open $s$-section $N_e$ containing $e$. Remark~\ref{property of regular space} yields an open
neighbourhood $N_e '$ of $e$ such that $\overline{N_e '}\subset N_e$. Since $K$ is compact, there
is a finite subset $F \subset K$, such that $\{N_e' : e \in F\}$ covers $K$. By Remark~\ref{POU on
subspace}, there exist $\{x_{1,e} : e \in F \}$, $\{x_{2,e} : e \in F \} \subset C_c(E^1)$ such
that $\osupp(x_{i,e}) \subset N_e' \cap K$, for all $e \in F$, $i=1$, $2$, and $\sum_{e \in
F}x_{1,e}=x_1$, $\sum_{e \in F}x_{2,e}=x_2$. Proposition~\ref{property of reduced Toep rep} implies
that
\[
\pi(\langle x_1,x_2 \rangle_{C_0(E^0)})
    =\sum_{e,f \in F}\pi(\langle x_{1,e},x_{2,f} \rangle_{C_0(E^0)})
    =\sum_{e,f \in F}\psi(x_{1,e})^*\psi(x_{2,f})=\psi(x_1)^*\psi(x_2).
\]
Then $\Vert \psi(x) \Vert^2=\Vert \psi(x)^*\psi(x) \Vert=\Vert \pi(\langle x,x \rangle_{C_0(E^0)})
\Vert \leq \Vert \langle x,x \rangle_{C_0(E^0)} \Vert=\Vert x \Vert_{C_0(E^0)}^2$, for all $x \in C_c(E^1)$, so $\psi$ is bounded, and hence extends uniquely to a bounded linear map $\widetilde{\psi}$ on
$X(E)$. By continuity, $(\widetilde{\psi},\pi)$ is a Toeplitz representation of $X(E)$.
\end{proof}

\begin{rmk}\label{rmk:Trep restricts}
Given a Toeplitz representation $(\psi,\pi)$ of $X(E)$, the pair $(\psi\vert_{C_c(E^1)}, \pi)$ is a
Toeplitz representation of $E$. So Proposition~\ref{reduced Toep rep is Toep rep} implies that
$(\psi, \pi) \mapsto (\psi|_{C_c(E^1)}, \pi)$ is a bijection between Toeplitz representations of
$X(E)$ and Toeplitz representations of $E$, with inverse described by Proposition~\ref{reduced Toep
rep is Toep rep}.
\end{rmk}

\begin{proof}[Proof of Theorem~\ref{thm:mainToeplitz}]
Let $(i_X,i_A)$ be the universal Toeplitz representation of $X(E)$ in $\mathcal{T}(E)$. Then
$(i_1,i_0):=(i_X\vert_{C_c(E^1)},i_A)$ is a Toeplitz representation of $E$. Fix another Toeplitz
representation $(\psi,\pi)$ of $E$ in a $C^*$-algebra $B$. By Proposition~\ref{reduced Toep rep is
Toep rep}, $\psi$ extends to $\widetilde{\psi} : X(E) \to B$ such that $(\widetilde{\psi},\pi)$ is
a Toeplitz representation of $X(E)$. By the universal property of $(i_X,i_A)$, there exists a
homomorphism $h_{\widetilde{\psi}, \pi} : \mathcal{T}(E) \to B$, such that $h_{\widetilde{\psi},
\pi} \circ i_X=\widetilde{\psi}$, and $h_{\widetilde{\psi}, \pi} \circ i_A=\pi$. In particular
$h_{\widetilde{\psi},\pi} \circ i_1=\psi$. Hence $(i_1,i_0)$ is a universal Toeplitz representation
of $E$ which generates $\mathcal{T}(E)$. The second statement follows easily.
\end{proof}

Our next task is to prove Theorem~\ref{thm:mainCK}. We first need some background results.

\begin{rmk}\label{alternative definition of E_{fin}^0}
Let $E$ be a topological graph. Fix $v \in E_{\mathrm{fin}}^0$. There exists a neighbourhood $U$ of
$v$ such that $r^{-1}(U)$ is compact, and there exists an open neighbourhood $V$ of $v$ such that
$V \subset U$. By Remark~\ref{property of regular space}, there exists an open neighbourhood $W$ of
$v$ such that $\overline{W} \subset V$. Since $r^{-1}(\overline{W})$ is closed and is contained in
$r^{-1}(U)$, it is compact. Hence $v \in E_{\mathrm{fin}}^0$ if and only if there exists an open
neighbourhood $N$ of $v$ such that $r^{-1}(\overline{N})$ is compact.
\end{rmk}

Let $E$ be a topological graph. Recall from Definition~\ref{dfn:O_X} that Katsura's covariance
condition for Toeplitz representations of $X(E)$ involves the ideal $\phi^{-1}(\mathcal{K}(X(E)))
\cap (\ker \phi)^{\perp}$. Katsura computed this ideal in \cite{Katsura:TAMS04}. We quote his
result and give a simple proof.

Observe that $\ker\phi=\big\{f \in C_0(E^0) : f\big(\overline{r(E^1)}\big)\equiv 0 \big\}$. Hence
$(\ker\phi)^{\perp}=\big\{f \in C_0(E^0) : f\big( \overline{E_{\mathrm{sce}}^0} \big) \equiv 0
\big\}$.

\begin{prop}[{\cite[Proposition~1.24]{Katsura:TAMS04}}]\label{compute of phi^{-1}(K(X(E)))}
Let $E$ be a topological graph. Then
\[
    \phi^{-1}(\mathcal{K}(X(E))) = C_0(E_{\mathrm{fin}}^0).
\]
Moreover $\phi^{-1}(\mathcal{K}(X(E))) \cap (\ker \phi)^{\perp}=C_0(E_{\mathrm{rg}}^0)$.
\end{prop}
\begin{proof}
The final statement follows from the previous one and definition of $E_{\mathrm{rg}}^0$. So we just
need to show that $\phi^{-1}(\mathcal{K}(X(E))) = C_0(E_{\mathrm{fin}}^0)$.

Fix $f \in C_0(E^0) \setminus C_0(E_{\mathrm{fin}}^0)$. We must show that $\phi(f) \not\in
\Kk(X(E))$. Fix $v_0 \in (E_{\mathrm{fin}}^0)^ c$, such that $f(v_0) \neq 0$. Let $m=\vert f(v_0)
\vert$ and let $N_0=\{v \in E^0:\ \vert f(v) \vert > m/2 \}$, so $N_0$ is an open neighbourhood of
$v_0$. By Remark~\ref{alternative definition of E_{fin}^0}, $r^{-1}(\overline{N_0})$ is not
compact. Fix $x_1, \dots, x_n$ and $y_1, \dots, y_n$ in $C_c(E^1)$. Let $K=\bigcup^n_{i=1}
\supp(x_i)\cup\supp(y_i)$. Then $K$ is compact, so $r^{-1}(\overline{N_0})$ is not contained in
$K$. So there exists $e_0 \in r^{-1}(\overline{N_0}) \setminus K$. By Remark~\ref{property of
regular space} there exists $x_0 \in C_c(E^1)$ such that $x_0(e_0)=1$. Hence
\begin{align*}
\Big\| \phi(f)-\sum^n_{i=1}\Theta_{x_i,y_i} \Big\|
    &\geq \Big\| \phi(f)x_0-\sum^n_{i = 1}\Theta_{x_i,y_i}(x_0) \Big\|_{C_0(E^0)} \\
    &\geq \Big\vert \Big\langle \phi(f)x_0-\sum^n_{i=1}\Theta_{x_i,y_i}(x_0), \phi(f)x_0-\sum^n_{i=1}\Theta_{x_i,y_i}(x_0) \Big\rangle_{C_0(E^0)} (s(e_0)) \Big\vert^{1/2} \\
    &\geq m/2.
\end{align*}
Thus $\|\phi(f) - a\| \geq m/2$ for all $a \in \lsp\{\Theta_{x,y} : x,y \in C_c(E^1)\}$. Since
$\clsp\{\Theta_{x,y} : x,y \in C_c(E^1)\} = \Kk(X(E))$, it follows that $\phi(f) \not\in
\Kk(X(E))$.

Now fix a nonnegative function $f \in C_c(E_{\mathrm{fin}}^0)$. We must show that $\phi(f) \in
\Kk(X(E))$. Let $K'=\supp(f)$. For any $e \in r^{-1}(K')$, there exists an open $s$-section $N_e$
containing $e$. Remark~\ref{properties of special subsets of E^0} shows that $r^{-1}(K')$ is
compact. Hence there exists a finite subset $F \subset r^{-1}(K')$ such that $\{N_e\}_{e \in F}$
covers $r^{-1}(K')$. Since $\supp(f \circ r) \subset r^{-1}(K')$, Remark~\ref{POU on subspace}
yields a finite collection of functions $\{f_e : e \in F\} \subset C_c(E^1)$ such that each
$\osupp(f_e) \subset N_e \cap r^{-1}(K')$ and $\sum_{e \in F} f_e=f \circ r$. Since the $f_e$ are
supported on the $s$-sections $N_e$, we have $\theta_{\sqrt{f_e}, \sqrt{f_e}}(x)(e') = f_e(e')
x(e')$ for all $e' \in E^1$. Hence
\begin{equation}\label{CK-relation of CP-algebra}
    \phi(f)=\sum_{e \in F}\Theta_{\sqrt{f_e},\sqrt{f_e}} \in \Kk(X(E)).\qedhere
\end{equation}
\end{proof}

\begin{prop}\label{reduced covariant Toep rep is covariant Toep rep}
Let $E$ be a topological graph and let $(\psi,\pi)$ be a covariant Toeplitz representation of $E$.
Then the Toeplitz representation $(\widetilde{\psi},\pi)$ of $X(E)$ from Proposition~\ref{reduced
Toep rep is Toep rep} is also covariant. Moreover, for any nonnegative function $f \in
C_c(E_{\mathrm{rg}}^0)$, any finite cover $\mathcal{N}$ of $r^{-1}(\supp(f))$ by open $s$-sections
and any collection of functions $\{f_N : N \in \mathcal{N}\} \subset C_c(E^1)$, such that
$\osupp(f_N) \subset N \cap r^{-1}(\supp(f))$, and $\sum_{N \in \mathcal{N}}f_N=f \circ r$, we have
$\pi(f)=\sum_{N \in \mathcal{N}_f}\psi(\sqrt{f_N})\psi(\sqrt{f_N})^*$.
\end{prop}
\begin{proof}
Let $\mathcal{G}$, $\{\mathcal{N}_f : f \in \mathcal{G}\}$, and $\{f_N : N \in \mathcal{N}_f\}$ be
as in Definition~\ref{dfn:new covariance}. By Corollary~\ref{compute of phi^{-1}(K(X(E)))}, to
prove $(\widetilde{\psi},\pi)$ is a covariant Toeplitz representation of $X(E)$, must show that
$\widetilde{\psi}^{(1)}\circ \phi(f)=\pi(f)$, for all $f \in \mathcal{G}$. Fix $f \in \mathcal{G}$,
since $\phi(f)=\sum_{N \in \mathcal{N}_f}\Theta_{\sqrt{f_N},\sqrt{f_N}}$, we have
$\widetilde{\psi}^{(1)}\circ \phi(f)=\sum_{N \in \mathcal{N}_f}\psi(\sqrt{f_N})\psi(\sqrt{f_N})^*$.
Hence $\widetilde{\psi}^{(1)}\circ \phi(f)=\pi(f)$ by Definition~\ref{dfn:new covariance}.
Therefore $(\widetilde{\psi},\pi)$ is a covariant Toeplitz representation of $X(E)$.

For the second statement observe that since $(\widetilde{\psi},\pi)$ is a covariant Toeplitz
representation of $X(E)$,
\[
\pi(f)
    =\widetilde{\psi}^{(1)}\circ \phi(f)
    =\sum_{N \in \mathcal{N}_f}\widetilde{\psi}^{(1)}(\Theta_{\sqrt{f_N},\sqrt{f_N}})
    =\sum_{N \in \mathcal{N}_f}\psi(\sqrt{f_N})\psi(\sqrt{f_N})^*.\qedhere
\]
\end{proof}

\begin{prop}\label{the restriction of covariant Toep rep is reduced covariant Toep rep}
Let $E$ be a topological graph and let $(\psi,\pi)$ be a covariant Toeplitz representation of $X(E)$. Then $(\psi\vert_{C_c(E^1)},\pi)$ is a covariant Toeplitz representation of $E$.
\end{prop}
\begin{proof}
Remark~\ref{rmk:Trep restricts} implies that $(\psi\vert_{C_c(E^1)},\pi)$ is a Toeplitz
representation of $E$. Let $\mathcal{G}$ be the set of all nonnegative functions in
$C_c(E_{\mathrm{rg}}^0)$. Fix $f \in \mathcal{G}$. By Equation~\eqref{CK-relation of CP-algebra},
there exists a finite cover $\mathcal{N}$ of $r^{-1}(\supp(f))$ by open $s$-sections and a finite
collection of functions $\{f_N : N \in \mathcal{N}\} \subset C_c(E^1)$, such that $\osupp(f_N)
\subset N \cap r^{-1}(\supp(f))$, $\sum_{N \in \mathcal{N}}f_N=f \circ r$, and $\phi(f)=\sum_{N \in
\mathcal{N}}\Theta_{\sqrt{f_N},\sqrt{f_N}}$. Since $(\psi,\pi)$ is a covariant Toeplitz
representation of $X(E)$, we have
\[
\pi(f)=\psi^{(1)} \circ \phi(f)=\sum_{N \in \mathcal{N}}\psi(\sqrt{f_N})\psi(\sqrt{f_N})^*=\sum_{N \in \mathcal{N}}\psi\vert_{C_c(E^1)}(\sqrt{f_N})\psi\vert_{C_c(E^1)}(\sqrt{f_N})^* .
\]
Hence $(\psi\vert_{C_c(E^1)},\pi)$ is covariant.
\end{proof}

\begin{proof}[Proof of Theorem~\ref{thm:mainCK}]
Propositions~\ref{reduced covariant Toep rep is covariant Toep rep} and \ref{the restriction of
covariant Toep rep is reduced covariant Toep rep} provide a bijective map from covariant Toeplitz
representations of $E$ to covariant Toeplitz representations of $X(E)$. The result now follows from
the same argument as Theorem~\ref{thm:mainToeplitz}.
\end{proof}

We now just need to prove Corollary~\ref{covariant condition of top graph with local homeomorphic
range map}. We must first show that under the hypotheses of the corollary, there are plenty of
local $r$-fibrations (see Definition~\ref{def:r-fibration}).

\begin{lemma}\label{existence of precompact open local r-fibration}
Let $E$ be a topological graph. Suppose that $r$ is a local homeomorphism and $r(E^1)$ is closed.
Then for any $v \in E_{\mathrm{rg}}^0$, there exists a precompact open local $r$-fibration
$(\mathcal{U},V)$, such that $v \in V \subset \overline{V} \subset E_{\mathrm{rg}}^0$.
\end{lemma}
\begin{proof}
Since $r$ is a local homeomorphism, it is in particular an open map (see, for example,
\cite[Page~4289]{Katsura:TAMS04}). Thus $r(E^1)$ is open in $E^0$. Since $r(E^1)$ is also closed,
the interior of $\overline{r(E^1)}$ is exactly $r(E^1)$. By Remark~\ref{properties of special
subsets of E^0},
\begin{equation}\label{E_rg^0 when r is a local homeomorphism with closed range}
 E_{\mathrm{rg}}^0=E_{\mathrm{fin}}^0 \cap r(E^1).
\end{equation}
Fix $v \in E_{\mathrm{rg}}^0$. By Remark~\ref{properties of special subsets of E^0}, $r^{-1}(v)$ is
a nonempty compact subset of $E^1$. Since $r$ is a local homeomorphism, $r^{-1}(v)$ is a finite set. Since $E^1$ is locally compact Hausdorff, we can separate points in $r^{-1}(v)$ by mutually disjoint precompact open bisections $\{U_e:\ e \in r^{-1}(v)\}$. We can assume, by shrinking if necessary, that $U_e$ have common range $N$, such that $\overline{N} \subset E_{\mathrm{rg}}^0$. We have $\vert r^{-1}(w)\vert \geq \vert r^{-1}(v)\vert$ for all $w \in N$.

We claim that there exists an open neighbourhood $V$ of $v$ such that $V \subset N$, and $\vert
r^{-1}(w)\vert =\vert r^{-1}(v)\vert$ for all $w \in V$. Suppose for a contradiction that there
exists a convergent net  $(v_\alpha)_{\alpha \in \Lambda} \subset N$ with limit $v$ satisfying
$\vert r^{-1}(v_\alpha)\vert> \vert r^{-1}(v)\vert$ for all $\alpha \in \Lambda$. So for any
$\alpha \in \Lambda$, there exists $e_\alpha \notin \bigcup_{e \in r^{-1}(v)}U_e$, such that
$r(e_\alpha)=v_\alpha$. Since $r^{-1}(\overline{N})$ is compact,
\cite[Theorem~IV.3]{ReedSimon:FunctAnal} implies that there exists a convergent subnet $(e_a')_{a
\in A}$ of $(e_\alpha)_{\alpha \in \Lambda}$ with the limit $e'$ not in $\bigcup_{e \in
r^{-1}(v)}U_e$. By the continuity of $r$, we have $r(e')=v$, which is a contradiction.

Hence there exists an open neighbourhood $V$ of $v$ satisfying $V \subset N$, such that $\vert
r^{-1}(w)\vert =\vert r^{-1}(v)\vert$ for all $w \in V$. So with $\mathcal{U}=\{ U_e \cap r^{-1}(V)
: e \in r^{-1}(v)\}$, the pair $(\mathcal{U},V)$ is a precompact open local $r$-fibration.
\end{proof}

\begin{lemma}\label{nonnegative functions supported by local r-fibration generates C_0(E_rg^0)}
Let $E$ be a topological graph. Suppose that $r$ is a local homeomorphism and $r(E^1)$ is closed.
Let $\mathcal{G}$ be the set of all nonnegative functions $f$ in $C_c(E_{\mathrm{rg}}^0)$ such that
$\supp(f) \subset V$ for some open local $r$-fibration $(\mathcal{U},V)$. Then $\mathcal{G}$
generates $C_0(E_{\mathrm{rg}}^0)$.
\end{lemma}
\begin{proof}
In order to prove $\mathcal{G}$ generates $C_0(E_{\mathrm{rg}}^0)$, it suffices to show the linear
span of $\mathcal{G}$ is $C_c(E_{\mathrm{rg}}^0)$. Fix $f \in C_c(E_{\mathrm{rg}}^0)$. By
Lemma~\ref{existence of precompact open local r-fibration}, for any $v \in \supp(f)$, there exists
an open local $r$-fibration $(\mathcal{U}_v,V_v)$, such that $v \in V_v$. Remark~\ref{POU
on subspace} yields a finite subset $F \subset \supp(f)$ and a finite collection of functions
$\{f_v : v \in F\} \subset C_c(E_{\mathrm{rg}}^0)$, such that $\{V_v : v \in F\}$ is an open cover
of $\supp(f)$, $\supp(f_v) \subset V_v$, for all $v \in F$, and $\sum_{v \in F}f_v=f$.
\end{proof}

\begin{proof}[Proof of Corollary~\ref{covariant condition of top graph with local homeomorphic range map}]
Let $(\psi,\pi)$ be a covariant Toeplitz representation of $E$. Let $\mathcal{G}$ be the set of all
nonnegative functions $f$ in $C_c(E_{\mathrm{rg}}^0)$ such that $\supp(f) \subset V$ for some open
local $r$-fibration $(\mathcal{U},V)$. Then Lemma~\ref{nonnegative functions supported by local
r-fibration generates C_0(E_rg^0)} implies that $\mathcal{G}$ generates $C_0(E_{\mathrm{rg}}^0)$.
Fix $f \in \mathcal{G}$ and an open local $r$-fibration $(\mathcal{U},V)$ with $\supp(f) \subset
V$. Since $\mathcal{U}$ is a finite cover of $r^{-1}(\supp(f))$ by open $s$-sections, each
$\osupp(r^*_U f) \subset U \cap r^{-1}(\supp(f))$, and $\sum_{U \in \mathcal{U}}r^*_U f=f \circ r$.
By Proposition~\ref{reduced covariant Toep rep is covariant Toep rep}, we have $\pi(f)=\sum_{U \in
\mathcal{U}}\psi(\sqrt{r^*_U f})\psi(\sqrt{r^*_U f})^*$. The second statement follows easily from
the construction of $\mathcal{G}$. The converse of the first statement follows from
Definition~\ref{dfn:new covariance} and Remark~\ref{converse statement proof of covariant
condition}.
\end{proof}

\subsection{The \texorpdfstring{$C^*$}{C*}-algebra generated by a Toeplitz representation}

In this subsection we provide some technical results which may prove useful in using our
descriptions of the $C^*$-algebras associated to topological graphs. Proposition~\ref{cor:checkable
toeplitz} is intended as an aid to constructing representations; and Proposition~\ref{prp:spanning
elements} provides a well-behaved collection of spanning elements for the image of any Toeplitz
representation of $E$, and also a formula for computing products of these spanning elements.

To construct Toeplitz representations of a topological graph, one needs to build linear maps $\psi:
C_c(E^1) \to B$ that are bounded in the bimodule norm $\Vert\cdot\Vert_{C_0(E^0)}$. The following
technical result simplifies the task by showing that it is enough to define $\psi$ on functions
that are dense in supremum norm on $C_0(U)$ for a suitable family of open $s$-sections $U$.

\begin{prop}\label{cor:checkable toeplitz}
Let $E$ be a topological graph, let $\mathcal{B}$ be an open base for the topology on $E^1$
consisting of $s$-sections, and let $\mathcal{F} \subset C_c(E^1)$ be a collection of nonnegative
functions such that $\osupp(x)$ is an $s$-section for all $x \in \mathcal{F}$. Suppose that
$\mathcal{G} \subset C_0(E^0)$ generates $C_0(E^0)$, and that for each $U \in \mathcal{B}$,
\begin{equation}\label{dense in sup norm implies dense in C_0(E^0) norm}
\lsp\{x \in \mathcal{F} : \osupp(x) \subset U\} \text{ is dense in } C_0(U) \text{ under the supremum norm}.
\end{equation}
Then $X_0 := \lsp\mathcal{F}$ is dense in $X(E)$. Let $B$ be a $C^*$-algebra. Suppose that $\psi_0:
X_0 \to B$ is a linear map, that $\pi : C_0(E^0) \to B$ is a homomorphism, and that
\begin{equation}\label{pi(<x,y>)=psi(x)^*psi(y)}
\pi(\widehat{\sqrt{xy}})=\psi_0(x)^*\psi_0(y)\quad\text{ for all $x,y \in \mathcal{F}$ (the product in $C_c(E^1)$ is pointwise)}.
\end{equation}
Then $\psi_0$ extends uniquely to a bounded linear map $\psi$ on $C_c(E^1)$. If the extension
$\psi$ satisfies
\begin{equation}\label{psi(fx)=pi(f)psi(x) for f in G x in F}
\psi(f \cdot x)= \pi(f)\psi_0(x)\quad\text{for all $f \in \mathcal{G}$ and $x \in \mathcal{F}$,}
\end{equation}
then $(\psi, \pi)$ is a Toeplitz representation of $E$.
\end{prop}
\begin{proof}
Fix $x \in C_c(E^1)$. Let $K=\supp(x)$. For each $e \in K$, there exists an open $s$-section $N_e$
containing $e$, such that $N_e \in \mathcal{B}$. Since $K$ is compact, there is a finite subset $F
\subset K$, such that $\{N_e : e \in F\}$ covers $K$. By Remark~\ref{POU on subspace}, there exists
a finite collection of functions $\{x_e : e \in F\} \subset C_c(E^1)$, such that $\osupp(x_e)
\subset N_e \cap K$ for all $e \in F$, and $\sum_{e \in F}x_e=x$. Fix $e \in F$. Since $\osupp(x_e)
\subset N_e$, there exists a sequence $(x_{e,n}) \subset X_0 \cap C_0(N_e)$ converging to $x_e$ in
supremum norm. That $(x_{e,n})$ and $x_e$ vanish off the $s$-section $N_e$ imply that $\Vert
x_{e,n}-x_e \Vert_{C_0(E^0)}=\sup_{e \in E^1}\vert x_{e,n}-x_e \vert$. Hence $\sum_{e \in F}x_{e,n}
\to x$ in $\Vert\cdot\Vert_{C_0(E^0)}$ norm. Therefore $X_0$ is dense in $X(E)$.

Fix $x$, $y \in \mathcal{F}$. Since $x$, $y$ are nonnegative, $\widehat{\sqrt{xy}}=\langle x,y
\rangle_{C_0(E^0)}$. Hence~\eqref{pi(<x,y>)=psi(x)^*psi(y)} implies that $\pi(\langle x,y
\rangle_{C_0(E^0)})=\psi_0(x)^*\psi_0(y)$. Linearity of $\psi_0$ and $\pi$ gives $\pi(\langle x,y
\rangle_{C_0(E^0)})=\psi_0(x)^*\psi_0(y)$, for all $x$, $y \in X_0$. By Remark~\ref{property of
Toep rep}, $\psi_0$ is bounded, and hence extends uniquely to a bounded linear map $\tilde{\psi}$
on $X(E)$. Then $\psi := \tilde{\psi}|_{C_c(E^1)}$ is the required map.

Equation~\eqref{psi(fx)=pi(f)psi(x) for f in G x in F} and continuity imply that $(\psi,\pi)$ is a
Toeplitz representation of $E$.
\end{proof}

\begin{rmk}
To prove Proposition~\ref{cor:checkable toeplitz} we showed that Equation~\eqref{dense in sup norm
implies dense in C_0(E^0) norm} implies that $X_0$ is dense in $X(E)$ under the
$\Vert\cdot\Vert_{C_0(E^0)}$ norm, and then deduced that $(\psi, \pi)$ extends to a Toeplitz
representation of $E$. So replacing Equation~\eqref{dense in sup norm implies dense in C_0(E^0)
norm} with the hypothesis that $X_0$ is dense in $X(E)$ would yield a formally stronger result.
However, Equation~\eqref{dense in sup norm implies dense in C_0(E^0) norm} is in many instances
easier to check.
\end{rmk}

Our next proposition provides a description of the $C^*$-algebra generated by a Toeplitz
representation of $E$ in terms of a spanning family which captures many of the key properties of
the usual spanning family in the Toeplitz algebra of a directed graph. We first need some notation
and two technical lemmas.

Recall that $E^n$ denotes the space $\{\mu = \mu_1 \dots \mu_n : \mu_i \in E^1, s(\mu_i) =
r(\mu_{i+1})\}$ of paths of length $n$ in $E$. We define $r,s : E^n \to E^0$ by $r(\mu) = r(\mu_1)$
and $s(\mu) = s(\mu_n)$, and we give $E^n$ the relative topology inherited from the product space
$\prod^n_{i=1} E^1$. For $x \in C_c(E^n)$ and $f \in C_0(E^0)$ we define $f \cdot x, x \cdot f \in
C_c(E^n)$ by $(f\cdot x)(\mu) = f(r(\mu)) x(\mu)$ and $(x \cdot f)(\mu) = x(\mu) f(s(\mu))$.

For $x_1, \dots, x_n \in C_c(E^1)$, we define $x_1 \diamond \dots \diamond x_n \in C_c(E^n)$ by
\[
(x_1 \diamond \dots \diamond x_n)(\mu) = \prod^n_{i=1} x_i(\mu_i) \quad\text{ for $\mu = \mu_1 \dots \mu_n \in E^n$.}
\]
We use the symbol $\diamond$ for this operation to distinguish it from the pointwise product of
elements of $C_c(E^1)$ appearing in, for example, Equation~\eqref{pi(<x,y>)=psi(x)^*psi(y)}.

The second assertion of the following technical result follows from the discussion preceding
\cite[Proposition~3.3]{Pimsner:FIC97} together with \cite[Proposition~1.27]{Katsura:TAMS04} (see
also \cite[Proposition~9.7]{Raeburn:GraphAlg}). We include the result and a simple proof here for
completeness.

Suppose that $x, y \in C_c(E^n)$ are supported on $s$-sections. Then there is a unique $H(x,y) \in
C_c(E^0)$ that vanishes on $E^0 \setminus \{s(\mu) : x(\mu)y(\mu) \not= 0\}$ and satisfies
\[
    H(x,y)(s(\mu)) = \overline{x(\mu)}y(\mu)\text{ whenever $x(\mu)y(\mu) \not= 0$}.
\]

\begin{lemma}\label{lem:psi^n}
Let $E$ be a topological graph and suppose that $x_1, \dots, x_n$ and $y_1, \dots, y_n$ are
supported on $s$-sections. Let $(\psi,\pi)$ be a Toeplitz representation of $E$. Let $x = x_1
\diamond \dots \diamond x_n$ and $y = y_1 \diamond \dots \diamond y_n$. Then $\pi(H(x,y)) =
\psi(x_n)^* \cdots \psi(x_1)^* \psi(y_1) \cdots \psi(y_n)$. If $x = y$ then $\prod^n_{i=1}
\psi(x_i) = \prod^n_{i=1} \psi(y_i)$.
\end{lemma}
\begin{proof}
By Proposition~\ref{reduced Toep rep is Toep rep}, $(\psi,\pi)$ extends to a Toeplitz
representation of $X(E)$. We have $H(x,y) = \langle x_n, H(x_1 \diamond \dots \diamond x_{n-1}, y_1
\diamond \dots \diamond y_{n-1})\cdot y_n\rangle$, and hence
\[
    \pi(H(x,y)) = \psi(x_n)^* \pi(H(x_1 \diamond \dots \diamond x_{n-1}, y_1 \diamond \dots \diamond y_{n-1}) \psi(y_n).
\]
The first assertion now follows by induction. For the second assertion, we use the first to see
that
\begin{align*}
\Big(\prod^n_{i=1} \psi(x_i) - \prod^n_{i=1} \psi(y_i)\Big)^*\Big(\prod^n_{i=1} \psi(x_i) - {}&\prod^n_{i=1} \psi(y_i)\Big) \\
    &= \pi(H(x,x) - H(x,y) - H(y,x) + H(y,y)),
\end{align*}
which is equal to zero since $x = y$.
\end{proof}

\begin{lemma}\label{lem:pass through}
Let $E$ be a topological graph. Suppose that $x_1, \dots, x_n \in C_c(E^1)$ are supported on
$s$-sections, and fix $f \in C_0(E^0)$. Then there exists $\tilde{f} \in C_c(E^0)$ such that
\begin{equation}\label{eq:pass f through}
\tilde{f}(s(\mu)) = f(r(\mu))\quad\text{ whenever $(x_1 \diamond \dots \diamond x_n)(\mu) \not= 0$.}
\end{equation}
For any such $\tilde{f}$, we have $f \cdot (x_1 \diamond \dots \diamond x_n) = (x_1 \diamond \dots
\diamond x_n)\cdot \tilde{f}$, and $\pi(f)\prod^n_{i=1} \psi(x_i) = \prod^n_{i=1}\psi(x_i)
\pi(\tilde{f})$ for any Toeplitz representation $(\psi, \pi)$ of $E$.
\end{lemma}
\begin{proof}
The second assertion follows immediately from the first by definition of $f \cdot (x_1 \diamond
\dots \diamond x_n)$ and $(x_1 \diamond \dots \diamond x_n)\cdot \tilde{f}$. The final assertion
then follows from Lemma~\ref{lem:psi^n}. So we just need to prove the first assertion. Let $x :=
x_1 \diamond \dots \diamond x_n$. Fix $f \in C_0(E^0)$. Since $K := \supp(x) \subseteq E^n$ is an
$s$-section, there is a well-defined continuous function from $s(K)$ to $r(K)$ given by $s(\mu)
\mapsto r(\mu)$ for $\mu \in K$. So there is a continuous function $\tilde{f}_0 \in C(s(K))$ given
by $\tilde{f}_0(s(\mu)) = f(r(\mu))$ for all $\mu \in K$. Since $s(K)$ is compact, an application
of the Tietze extension theorem shows that $\tilde{f}$ has an extension $\tilde{f} \in C_0(E^0)$,
which satisfies~\eqref{eq:pass f through} by definition.
\end{proof}

\begin{prop}\label{prp:spanning elements}
Let $E$ be a topological graph and let $(\psi,\pi)$ be a Toeplitz representation of $E$. Let
$\mathcal{B}$ and $\mathcal{F}$ be as in Proposition~\ref{cor:checkable toeplitz}, and suppose that
$\supp(x)$ is an $s$-section for each $x \in \mathcal{F}$. Then
\begin{enumerate}
\item\label{it:spanning} $C^*(\psi,\pi)$ is densely spanned by elements of the form
    \[
        \psi(x_1)\cdots\psi(x_n) \pi(f) \psi(y_m)^*\cdots\psi(y_1)^*
    \]
    where $m,n \ge 0$, $f \in C_c(E^0)$, the $x_i, y_j$ all belong to $\mathcal{F}$, each
    $s(\osupp(x_i)) \cap r(\osupp(x_{i+1})) \neq \emptyset$ and each $s(\osupp(y_i)) \cap
    r(\osupp(y_{i+1})) \neq \emptyset$, and $s(\osupp(x_n)) \cap s(\osupp(y_m)) \cap \osupp(f)
    \neq \emptyset$.
\item\label{it:mult} Let $\psi(w_1) \dots \psi(w_m)\pi(f)\psi(x_n)^* \dots \psi(x_1)^*$ and
    $\psi(y_1) \dots \psi(y_p)\pi(g) \psi(z_q)^* \dots \psi(z_1)^*$ be spanning elements as
    in~\eqref{it:spanning} with $p \ge n$. Let $x = x_1 \diamond \dots \diamond x_n$ and $y =
    y_1 \diamond \dots \diamond y_n$. Fix $k \in C_0(E^0)$ such that $(f H(x,y)) \cdot (y_{n+1}
    \diamond \dots \diamond y_p) = (y_{n+1} \diamond \dots \diamond y_p)\cdot k$ as in
    Lemma~\ref{lem:pass through}. Then
    \begin{align*}
    \big(\psi(w_1) \dots \psi(w_m)\,\pi(f)\,\psi&(x_n)^* \dots \psi(x_1)^*\big)
        \big(\psi(y_1) \dots \psi(y_p)\,\pi(g)\,\psi(z_q)^* \dots \psi(z_1)^*\big) \\
            & = \psi(w_1) \dots \psi(w_m) \psi(y_{n+1}) \dots \psi(y_p)\,
                \pi(k g)\,\psi(z_q)^* \dots \psi(z_1)^*.
    \end{align*}
\end{enumerate}
\end{prop}

\begin{rmk}
Consider the situation of Proposition~\ref{prp:spanning elements}\eqref{it:mult} but with $p < n$.
Let $x = x_1 \diamond \dots \diamond x_p$ and $y = y_1 \diamond \dots \diamond y_p$, and fix $k'
\in C_0(E^0)$ such that $(H(x,y) g) \cdot (x_{p+1} \diamond \dots \diamond x_n) = (x_{p+1} \diamond
\dots \diamond x_n) \cdot k'$. Taking adjoints in Proposition~\ref{prp:spanning
elements}\eqref{it:mult} gives
    \begin{align*}
    \big(\psi(w_1) \dots \psi(w_m)\,\pi(f)\,\psi&(x_n)^* \dots \psi(x_1)^*\big)
        \big(\psi(y_1) \dots \psi(y_p)\,\pi(g)\,\psi(z_q)^* \dots \psi(z_1)^*\big) \\
            & = \psi(w_1) \dots \psi(w_m)\,\pi(fk')\,\psi(x_n)^* \cdots \psi(x_{p+1})\psi(z_q)^* \dots \psi(z_1)^*.
    \end{align*}
\end{rmk}

\begin{proof}[Proof of Proposition~\ref{prp:spanning elements}]
\eqref{it:spanning} Proposition~\ref{reduced Toep rep is Toep rep} implies that
$(\widetilde{\psi},\pi)$ is a Toeplitz representation of $X(E)$, where $\widetilde{\psi}$ is the
unique extension of $\psi$ to $X(E)$. The argument of \cite[Proposition~2.7]{Katsura:JFA04} shows
that $C^*(\psi,\pi)$ is densely spanned by elements of the form
$\psi(x_1)\cdots\psi(x_n)\,\pi(f)\,\psi(y_m)^*\cdots\psi(y_1)^*$ where each $x_i, y_i \in C_c(E_1)$
and $f \in C_c(E^0)$. Fix $x_1$, $x_2 \in C_c(E^1)$ with $s(\osupp(x_1)) \cap r(\osupp(x_2)) =
\emptyset$. Then
\begin{align*}
\Vert \psi(x_1)\psi(x_2) \Vert^2
    &=\Vert \psi(x_2)^*\psi(x_1)^*\psi(x_1)\psi(x_2) \Vert
     =\Vert \psi(x_2)^*\pi(\langle x_1,x_1 \rangle_{C_0(E^0)}) \psi(x_2)\Vert \\
    &=\Vert \psi(x_2)^* \psi(\langle x_1,x_1 \rangle_{C_0(E^0)}\cdot x_2) \Vert
     =0.
\end{align*}
Similarly, $\psi(x_1)\psi(x_2)^*=0$ whenever $s(\osupp(x_1)) \cap s(\osupp(x_2))= \emptyset$. So
$C^*(\psi, \pi)$ is densely spanned by elements $\psi(x_1) \dots \psi(x_m) \pi(f) \psi(y_m)^* \dots
\psi(y_1)^*$ where $f \in C_c(E^0)$, each $s(\osupp(x_i)) \cap r(\osupp(x_{i+1})) \not= \emptyset$,
each $s(\osupp(y_i)) \cap r(\osupp(y_{i+1})) \not= \emptyset$, and $s(\osupp(x_n)) \cap
s(\osupp(y_m)) \cap \osupp(f) \neq \emptyset$. Since Proposition~\ref{cor:checkable toeplitz}
implies that $X_0 = \lsp{\mathcal{F}}$ is dense in $X(E)$ and hence in $C_c(E^1)$, the first
assertion follows.

\eqref{it:mult} Lemma~\ref{lem:psi^n} implies that
\[
    \pi(f)\psi(x_n)^* \dots \psi(x_1)^* \psi(y_1) \dots \psi(y_p)
        = \pi(f H(x,y)) \psi(y_{n+1}) \dots \psi(y_p).
\]
The result now follows from Lemma~\ref{lem:pass through}.
\end{proof}

\begin{eg}\label{eg:directed graph toeplitz}
Let $E$ be a directed graph regarded as a topological graph under the discrete topology. Let
$\mathcal{G} = \{\delta_v : v \in E^0\}$ and $\mathcal{F} = \{\delta_e : e \in E^1\}$. Then
$\mathcal{G}$ and $\mathcal{F}$ satisfy the hypotheses of Proposition~\ref{cor:checkable toeplitz},
so we recover as an immediate consequence the isomorphism of the Toeplitz algebra $\mathcal{T}(E)$
of the graph bimodule (see \cite{FowlerRaeburn:IUMJ99}~and~\cite{Katsura:TAMS04}) with the Toeplitz
algebra $\mathcal{T}C^*(E)$ of the graph $E$. The usual spanning family and multiplication rule for
$\mathcal{T}C^*(E)$ follows from Proposition~\ref{prp:spanning elements} applied to the same
$\mathcal{F}$.
\end{eg}

\begin{rmk}
The multiplication formula of Proposition~\ref{prp:spanning elements}\eqref{it:mult} has the
drawback that the element $k$ has no explicit formula in terms of the $x_i$ the $y_i$ and the
function $f$; it is obtained by an application of the Tietze extension theorem (see
Lemma~\ref{lem:pass through}). However, in practise there will frequently be a natural choice for
$k$. Suppose, for example, that $E^1$ is totally disconnected. Then $\mathcal{F}$ can be taken to
consist of characteristic functions of compact open $s$-sections. We can then take $k$ to be the
function that is identically zero off $s(\supp(y_{n+1} \diamond \dots \diamond y_p))$ and satisfies
$k(s(\mu)) = f(r(\mu)) H(x,y)(r(\mu))$ whenever $\mu \in \supp(y_{n+1} \diamond \dots \diamond
y_p)$; this $k$ is continuous because $s(\supp(y_{n+1} \diamond \dots \diamond y_p))$ is clopen.
\end{rmk}

\section{The topological graph arising from the shift map on the infinite path space}\label{sec:example}

In this section we discuss how our results apply to the topological graph $\widehat{E}$ arising
from the shift map on the infinite path space $E^\infty$ of a row-finite directed graph $E$ with no
sources. It is known that $\mathcal{O}_{X(\widehat{E})}$ is isomorphic to $C^*(E)$ (it could be
recovered from \cite{Brownlowe, BRV}) but existing proofs use the universal property of $C^*(E)$ to
induce a homomorphism from $C^*(E)$ to $\mathcal{O}_{X(\widehat{E})}$, invokes the gauge-invariant
uniqueness theorem for $C^*(E)$ to establish injectivity, and then argues surjectivity by hand. It
takes some work to show using the universal property of $\Oo_{X(\widehat{E})}$ that there is a
homomorphism going in the other way.

Let $E$ be a row-finite directed graph with no sources. We define $E^*=\bigcup_{n \geq 0}E^n$ and
define $E^{\infty}=\big\{z \in \prod_{i=1}^{\infty}E^1 : s(z_i)=r(z_{i+1}), \text{ for all $i=1, 2,
\dots$}\big\}$. We view $E^{\infty}$ as a topological space under the subspace topology coming from
the ambient space $\prod_{i=1}^{\infty}E^1$. For any $\mu \in E^* \setminus E^0$, we define the
cylinder set $Z(\mu)=\{z \in E^{\infty}: z_1=\mu_1, \dots, z_{\vert\mu\vert}=\mu_{\vert\mu\vert}
\}$. For $v \in E^0$ we define $Z(v)= \{z \in E^\infty : r(z_1) = v\}$. Since $E$ has no sources,
each $Z(\mu)$ is nonempty. The space $E^{\infty}$ is a locally compact Hausdorff space with a base
of compact open sets $\{Z(\mu) : \mu \in E^*\}$ (\cite[Corollary~2.2]{KumjianPaskEtAl:JFA97}).

Now we construct a topological graph
$\widehat{E}=(\widehat{E}^0,\widehat{E}^1,\widehat{r},\widehat{s})$ as follows. Let $\widehat{E}^0
= \widehat{E}^1=E^\infty$. Define $\widehat{r}$ to be the identity map, and define
$\widehat{s}(z)=(z_2,z_3,\dots)$ for all $z \in \widehat{E}^1$. Since $Z(\mu)$ is a compact open
$\widehat{s}$-section whenever $\mu \not\in E^0$, the map $\widehat{s}$ is a local homeomorphism
and hence $\widehat{E}$ is a topological graph.

For the following result recall the definition of a Cuntz-Krieger $E$-family from the first
paragraph of the introduction.

\begin{prop}\label{prp:Ehat}
Let $E$ be a row-finite directed graph with no sources, and let $\widehat{E}$ be the topological
graph described above.
\begin{enumerate}
\item\label{it:repn to CK} Let $(\psi,\pi)$ be a covariant Toeplitz representation of
    $\widehat{E}$. For $v \in E^0$ define $q_v := \pi(\chi_{Z(v)})$ and for $e \in E^1$ define
    $t_e := \psi(\chi_{Z(e)})$. Then the $q_v$ and the $t_e$ form a Cuntz-Krieger $E$-family.
\item\label{it:CK to repn} Let $\{q_v : v \in E^0\}$, $\{t_e : e \in E^1\}$ be a Cuntz-Krieger
    $E$-family in a $C^*$-algebra $B$. Then there is a unique covariant Toeplitz representation $(\psi,
    \pi)$ of $\widehat{E}$ such that $\psi(\chi_{Z(e\mu)}) = t_{e\mu} t^*_\mu$, and $\pi(\chi_{Z(\mu)}) = t_\mu t^*_\mu$ for all $e \in E^1$, $\mu \in E^*$.
\item\label{isomorphism between O(widehat{E}) and C^*(E)} Let $(j_1,j_0)$ be the universal Toeplitz representation of $\widehat{E}$ in $\mathcal{O}(E)$, and let $\{p_v,\ s_e:\ v \in E^0,\ e \in E^1\}$ be the Cuntz-Krieger $E$-family generating $C^*(E)$. Then there is an isomorphism $\mathcal{O}(\widehat{E}) \cong C^*(E)$ which carries each $j_1(\chi_{Z(e\mu)})$ to $s_{e\mu} s^*_\mu$, and each $j_0(\chi_{Z(\mu)})$ to $s_\mu s^*_\mu$.
\end{enumerate}
\end{prop}
\begin{proof}[Proof of Proposition~\ref{prp:Ehat}\eqref{it:repn to CK}]
The $q_v$ are mutually orthogonal projections because the $\chi_{Z(v)}$ are. For $\mu \in E^*
\setminus E^0$, the set $Z(\mu)$ is a compact open $s$-section. Thus for $e \in E^1$,
relation~\eqref{it:s*s} of Definition~\ref{reduced Toep rep} implies that
$t_e^*t_e=\pi(\widehat{\chi_{Z(e)}})=\pi(\chi_{Z(s(e))})=q_{s(e)}$. For $\mu \in E^* \setminus
E^0$, $(\{Z(\mu)\},Z(\mu))$ is an open local $\widehat{r}$-fibration. Since
$\supp(\chi_{Z(\mu)})=Z(\mu)$ and $(\psi, \pi)$ is covariant, Corollary~\ref{covariant condition of
top graph with local homeomorphic range map} implies that
\begin{equation}\label{covariant condition of E^hat}
\pi(\chi_{Z(\mu)})=\psi\Big(\sqrt{r^*_{Z(\mu)} \chi_{Z(\mu)}}\Big)\psi\Big(\sqrt{r^*_{Z(\mu)} \chi_{Z(\mu)}}\Big)^*=\psi(\chi_{Z(\mu)})\psi(\chi_{Z(\mu)})^*.
\end{equation}
So for $v \in E^0$, we have
\[
q_v = \pi(\chi_{Z(v)})
    = \sum_{r(e) = v} \pi(\chi_{Z(e)})
    = \sum_{r(e) = v} \psi(\chi_{Z(e)}) \psi(\chi_{Z(e)})^*
    = \sum_{r(e) = v} t_e t^*_e. \qedhere
\]
\end{proof}

\begin{proof}[Proof of Proposition~\ref{prp:Ehat}\eqref{it:CK to repn}]
Let $\Gg:= \{\chi_{Z(\mu)} : \mu \in E^* \setminus E^0\} \subseteq C_0(\widehat{E}^0)$. Then
$\lsp\Gg$ is a dense $*$-subalgebra of $C_0(\widehat{E}^0)$. We aim to define a map $\pi_0:
\lsp\mathcal{G} \to B$ by $\pi_0(\sum_{i=1}^{n}\alpha_i \chi_{Z(\mu_i)})=\sum_{i=1}^{n}\alpha_i
t_{\mu_i}t_{\mu_i}^*$. We check that $\pi_0$ is well-defined. It suffices to prove that
$\sum_{i=1}^{n}\alpha_i \chi_{Z(\mu_i)}=0$ implies $\sum_{i=1}^{n}\alpha_i t_{\mu_i}t_{\mu_i}^*=0$,
where the $\mu_i$ are distinct. Since $E$ is row-finite and has no sources,
\begin{align*}
\pi_0\Big(\sum_{e \in r^{-1}(s(\mu))}\chi_{Z(\mu e)}\Big)=\sum_{e \in r^{-1}(s(\mu))}t_{\mu_e}t_{\mu_e}^*=t_{\mu}\Big(\sum_{e \in r^{-1}(s(\mu))}t_et_e^*\Big)t_{\mu}^*=t_\mu t_\mu^*=\pi_0(\chi_{Z(\mu)}),
\end{align*}
so we can assume that the $\mu_i$ have the same length. It follows that the $\chi_{Z(\mu_i)}$ are
mutually orthogonal nonzero projections and hence each $\alpha_i=0$. So $\pi_0$ is well-defined. It
is obvious that $\pi_0$ is a linear map preserving the involution.
\cite[Corollary~1.15]{Raeburn:GraphAlg} implies that $\pi_0$ is a homomorphism. Now we show that
$\pi_0$ is norm decreasing. Fix $\mu_1$, $\dots$, $\mu_n$. We can assume that the $\mu_i$ are
distinct and have the same length. Then
\begin{align*}
\big\Vert \pi_0(\sum_{i=1}^{n}\alpha_i \chi_{Z(\mu_i)} ) \big\Vert^2&=\big\Vert \sum_{i=1}^{n} \vert\alpha_i\vert^2 t_{\mu_i}t_{\mu_i}^* \big\Vert \leq (\max_i\vert\alpha_i\vert^2)\big\Vert \sum_{i=1}^{n} t_{\mu_i}t_{\mu_i}^* \big\Vert \leq \max_i\vert\alpha_i\vert^2
\\&=\big\Vert\sum_{i=1}^{n}\alpha_i \chi_{Z(\mu_i)}\big\Vert^2,
\end{align*}
so $\pi_0$ is norm decreasing. Thus we obtain a unique homomorphism $\pi:C_0(\widehat{E}^0) \to C^*(E)$ by $\pi(\chi_{Z(\mu)})=t_{\mu}t_{\mu}^*$ for all $\mu \in E^*$.

We next aim to define a linear map $\psi : C_c(\widehat{E}^1) \to B$ by extension of the formula
$\psi(\chi_{Z(e\mu)}) = t_{e\mu} t^*_\mu$, and to show that the pair $(\psi, \pi)$ is a Toeplitz
representation of $\widehat{E}$. To do so, we will apply Proposition~\ref{cor:checkable
toeplitz}, so we need to set up the rest of the elements of the statement. Let $\Bb := \{Z(\mu) : \mu \in
E^* \setminus E^0\}$, and let $\Ff := \{\chi_{Z(e\mu)} : e \in E^1, \mu \in E^*\} \subseteq C_c(E^1)$. Certainly $\Ff$ and $\Bb$ satisfy Equation~\ref{dense in sup norm implies dense in C_0(E^0) norm}. Similarly to the construction of $\pi_0$, there is a well-defined linear map $\psi_0 : \lsp \Ff \to B$ satisfying $\psi_0(\sum^n_{i=1} \alpha_i \chi_{Z(e_i\mu_i)}) =\sum^n_{i=1} \alpha_i t_{e_i\mu_i}t^*_{\mu_i}$.

Fix $x = \chi_{Z(e\mu)}$ and $y =\chi_{Z(f\nu)}$ in $\Ff$. We verify Equation~\ref{pi(<x,y>)=psi(x)^*psi(y)}. For this, observe that
\[
\widehat{\sqrt{xy}} = \begin{cases}
    \chi_{Z(\mu)} &\text{ if $e\mu = f\nu\mu'$}\\
    \chi_{Z(\nu)} &\text{ if $f\nu = e\mu\nu'$}\\
    0 &\text{ otherwise.}
\end{cases}
\]
Then calculate:
\[
(t_{e\mu} t^*_\mu)^* t_{f\nu} t^*_\nu
    = t_\mu t_{e\mu}^* t_{f\nu} t^*_\nu
    = \begin{cases}
        t_\nu t^*_\nu &\text{ if $f\nu = e\mu\nu'$}\\
        t_\mu t^*_\mu &\text{ if $e\mu = f\nu\mu'$}\\
        0 &\text{ otherwise.}
    \end{cases}
\]
This establishes Equation~\eqref{pi(<x,y>)=psi(x)^*psi(y)}. So Proposition~\ref{cor:checkable
toeplitz} shows that $\psi_0$ extends uniquely to a linear map $\psi : C_c(\widehat{E}^1) \to B$. A
similar calculation establishes Equation~\eqref{psi(fx)=pi(f)psi(x) for f in G x in F}.
Proposition~\ref{cor:checkable toeplitz} implies that $(\psi, \pi)$ is a Toeplitz representation of
$E$.

It remains to check covariance. Since the range map is a homeomorphism onto $\widehat{E}^0$ we can
apply Corollary~\ref{covariant condition of top graph with local homeomorphic range map} with $\Gg$
as in the proof of part~\eqref{it:CK to repn} and the local $\widehat{r}$-fibrations
$\big\{(\{Z(\mu)\}, Z(\mu)) : \mu \in E^* \setminus E^0\}$ to see that $(\psi,\pi)$ is covariant.
\end{proof}

\begin{proof}[Proof of Proposition~\ref{prp:Ehat}\eqref{isomorphism between O(widehat{E}) and C^*(E)}]
We show that $C^*(E)$ has the universal property of $\mathcal{O}(\widehat{E})$ and then invoke
Theorem~\ref{thm:mainCK}. Proposition~\ref{prp:Ehat}\eqref{it:CK to repn} yields a covariant
Toeplitz representation $(\theta_1,\theta_0)$ of $\widehat{E}$ in $C^*(E)$ such that
$\theta_1(\chi_{Z(e\mu)}) = s_{e\mu} s^*_\mu$ for all $e \in E^1$, $\mu \in E^*$, and
$\theta_0(\chi_{Z(\mu)}) = s_\mu s^*_\mu$ for all $\mu \in E^*$. Fix a covariant Toeplitz
representation $(\psi,\pi)$ of $\widehat{E}$ in a $C^*$-algebra $B$. Then
Proposition~\ref{prp:Ehat}\eqref{it:repn to CK} gives a Cuntz-Krieger $E$-family
$\{q_v:=\pi(\chi_{Z(v)}),\ t_e:=\psi(\chi_{Z(e)}):\ v \in E^0,\ e \in E^1 \}$ in $B$. So
\cite[Proposition~1.21]{Raeburn:GraphAlg} gives a homomorphism $\rho: C^*(E) \to B$ such that
$\rho(p_v)=q_v$, and $\rho(s_e)=t_e$. An induction on the length of $\mu$ using
Equation~\eqref{covariant condition of E^hat} show that $\pi(\chi_{Z(\mu)})=t_\mu t_\mu^*$. For $e
\in E^1$, and $\mu \in E^*$, we have $\rho \circ \theta_1(\chi_{Z(e\mu)})=t_e t_\mu
t_\mu^*=\psi(\chi_{Z(e)})\pi(\chi_{Z(\mu)})=\psi(\chi_{Z(e\mu)})$. For $\mu \in E^* \setminus E^0$,
we have $\rho \circ \theta_0(\chi_{Z(\mu)})=t_\mu t_\mu^*=\pi(\chi_{Z(\mu)})$. Hence $\rho \circ
\theta_1=\psi$, and $\rho \circ \theta_0=\pi$. Since the image of $(\theta_1,\theta_0)$ generates
$C^*(E)$, Theorem~\ref{thm:mainCK} implies that there is an isomorphism $\mathcal{O}(E) \cong
C^*(E)$ which carries each $j_1(\chi_{Z(e\mu)})$ to $s_{e\mu} s^*_\mu$, and each
$j_0(\chi_{Z(\mu)})$ to $s_\mu s^*_\mu$.
\end{proof}

\end{document}